\documentclass[11pt]{article}
\usepackage[letterpaper,hmargin=1in,vmargin=1in]{geometry}
\usepackage{latexsym, amsfonts, amssymb, amsmath, amsthm, hyperref, multirow}

\def\g{\mathfrak{g}}

\def\P{{\sf P}}

\def\I{{\cal I}}
\def\S{{\cal R}}
\def\T{{\cal T}}
\def\ZZ{{\mathbb Z}}

\def\EE{E}

\newcommand{\bone}{{\bf 1}}
\renewcommand{\Pr}[1]{{\sf Pr}\left(#1\right)}

\newtheorem{theorem}{Theorem}[section]
\newtheorem{lemma}[theorem]{Lemma}

\newtheorem{definition}[theorem]{Definition}

\newtheorem{remark}[theorem]{Remark}

\title{How Long Does it Take to Catch a Wild Kangaroo?}

\author{Ravi Montenegro \and Prasad Tetali}

\author{
Ravi Montenegro
\thanks{Department of Mathematical Sciences, University of Massachusetts Lowell, Lowell, MA 01854, USA.
Email: ravi\_montenegro@uml.edu;
part of this work was done while the author was at The Tokyo Institute of Technology.
}
\and
Prasad Tetali
\thanks{School of Mathematics and College of Computing, Georgia Institute of Technology, Atlanta, GA 30332, USA.
Email: tetali@math.gatech.edu;
research supported in part by NSF grants DMS 0401239, 0701043.}
}

\begin{document}

\maketitle

\begin{abstract}
\noindent
We develop probabilistic tools for upper and lower bounding the expected time until two independent random walks on $\ZZ$ intersect each other.
This leads to the first sharp analysis of a non-trivial Birthday attack,
proving that Pollard's Kangaroo method solves the discrete logarithm problem $g^x=h$ on a cyclic group in expected time
$(2+o(1))\sqrt{b-a}$ for an average $x\in[a,b]$.
Our methods also resolve a conjecture of Pollard's, by showing that the same bound holds when step sizes are generalized from powers of $2$ to powers of any fixed $n$.
\end{abstract}

\section{Introduction}

Probabilistic ``paradoxes'' can have unexpected applications in computational problems,
but mathematical tools often do not exist to prove the reliability of the resulting computations, so instead practitioners have to rely on heuristics, intuition and experience.
A case in point is the Kruskal Count, a probabilistic concept discovered by Martin Kruskal and popularized in a card trick by Martin Gardner, which exploits the property that for many Markov chains on $\ZZ$ independent walks will intersect fairly quickly when started at nearby states.
In a 1978 paper John Pollard applied the same trick to a mathematical problem related to code breaking, the Discrete Logarithm Problem: solve for the exponent $x$, given the generator $g$ of a cyclic group $G$
and an element $h\in G$ such that $g^x=h$.

Pollard's Kangaroo method is based on running two independent random walks on a cyclic group $G$, one starting at a known state (the ``tame kangaroo'') and the other starting at the unknown but nearby value of the discrete logarithm $x$ (the ``wild kangaroo''), and terminates after the first intersection of the walks.
As such, in order to analyze the algorithm it suffices to develop probabilistic tools for examining the expected time until independent random walks on a cyclic group intersect, in terms of some measure of the initial distance between the walks.

Past work on problems related to the Kruskal Count seem to be of little help here.
Pollard's argument of \cite{Pol00.1} gives rigorous results for specific values of $(b-a)$, but the recurrence relations he uses can only be solved on a case-by-case basis by numerical computation.
Lagarias et.al. \cite{LRV09.1} used probabilistic methods to study the {\em distance traveled} before two walks intersect, but only for walks in which the number of steps until an intersection was simple to bound.
Although our approach here borrows a few concepts from the study of  the Rho algorithm in \cite{KMPT07.1}, such as examining the expected number of intersections and some measure of its variance,
a significant complication in studying this algorithm is that when $b-a\ll|G|$ the kangaroos will have proceeded only a small way around the cyclic group before the algorithm terminates.
As such, mixing time is no longer a useful notion, and instead a notion of convergence is required which occurs long before the mixing time.
The tools developed here to avoid this problem may prove of independent interest when examining other pre-mixing properties of Markov chains.

The key probabilistic results required are upper and lower bounds on expected time until intersection of independent walks on $\ZZ$ started from nearby states.
In the specific case of the walk involved in the Kangaroo method these bounds are equal, and so the lead constants are sharp, which is quite rare among the analysis of algorithms based on Markov chains.
More specifically we have:

\begin{theorem} \label{thm:main}
Suppose $g,h\in G$ are such that $h=g^x$ for some $x\in[a,b]$.
If $x$ is a uniform random integer in $[a,b]$ then the expected number of group operations required by the Distinguished Points implementation of Pollard's Kangaroo method is
$$
(2+o(1))\sqrt{b-a}\,.
$$
The expected number of group operations is maximized when $x=a$ or $x=b$, at
$$
(3+o(1))\sqrt{b-a}
$$
\end{theorem}

Pollard \cite{Pol00.1} previously gave a convincing but not completely rigorous argument for the first bound,
while the second was known only by a rough heuristic.
Given the practical significance of Pollard's Kangaroo method for solving the discrete logarithm problem, we find it surprising that there has been no fully rigorous analysis of this algorithm, particularly since it has been 30 years since it was first proposed in \cite{Pol78.1}.

The paper proceeds as follows.
A general framework for analyzing intersection of independent walks on the integers is constructed in Section \ref{sec:collision}.
This is followed in Section \ref{sec:prelim} by a detailed description of the Kangaroo method,
with analysis in Section \ref{sec:kangaroo}.
The paper finishes in Section \ref{sec:generalize} with an extension of the results to more general step sizes, resolving a conjecture of Pollard's.


\section{Uniform Intersection Time and a Collision Bound} \label{sec:collision}

Given two independent instances $X_i$ and $Y_j$ of a Markov Chain on $\ZZ$, started at nearby states $X_0$ and $Y_0$ (as made precise below), we consider the expected number of steps required by the walks until they first intersect.
Observe that if the walk is increasing, i.e. $\P(u,v)>0$ only if $v>u$, then to examine the number of steps required by the $X_i$ walk it suffices to let $Y_j$ proceed an infinite number of steps and then evolve $X_i$ until $X_i=Y_j$ for some $i,j$.
Thus, rather than considering a specific probability $\Pr{X_i=Y_j}$ it is better to look at $\Pr{\exists j:\,X_i=Y_j}$.
By symmetry, the same approach will also bound the expected number of steps required by $Y_j$ before it reaches a state visited by the $X_i$ walk.

First, however, because the walk is not ergodic then
alternate notions resembling mixing time and a stationary distribution will be required.
Heuristic suggests that after some warm-up period the $X_i$ walk will be sufficiently randomized that at each subsequent step the probability of colliding with the $Y_j$ walk is roughly the inverse of the average step size.
Our replacement for mixing time will measure the number of steps required for this to become a rigorous statement:

\begin{definition}
A {\em stopping time} for a random walk $\{X_i\}_{i=0}^{\infty}$ is a random variable $T\in{\mathbb N}$ such that the event $\{T=t\}$ depends only on $X_0,\,X_1,\,\ldots,\,X_t$.
The average time until stopping is $\overline{T}=\EE T$.
\end{definition}

\begin{definition} \label{def:uniform-time}
Consider a Markov chain $\P$ on an infinite group $G$.
A {\em nearly uniform intersection time} $T(\epsilon)$ is a stopping time such that for some $U>0$ and $\epsilon\geq 0$ the relation
$$
(1-\epsilon)U \leq \Pr{\exists j:\,X_{T(\epsilon)+\Delta}=Y_j} \leq (1+\epsilon)U
$$
holds for every $\Delta\geq 0$ and every $(X_0,Y_0)$ in a designated set of initial states $\Omega\subset G\times G$.
\end{definition}

In general the probability that two walks will ever intersect may go to zero in the limit.
However, if a walk is {\em transitive} on $\ZZ$ (i.e. $\P(u,v)=\P(0,v-u)$), {\em increasing} (i.e. $\P(u,v)>0$ only when $v>u$), and {\em aperiodic} (i.e. $gcd\{k:\,\P(0,k)>0\}=1$), 
then one out of every $\bar{S} = \sum_{k=1}^\infty k\P(0,k)$ states is visited and a stopping time will exist satisfying
$$
\frac{1-\epsilon}{\bar{S}} \leq \Pr{\exists j:\,X_{T(\epsilon)+\Delta}=Y_j} \leq \frac{1+\epsilon}{\bar{S}}\,.
$$
An obvious choice of starting states are all $Y_0\leq X_0$, but for reasons that will be apparent later it better serves our purposes to expand to the case of $Y_0<X_0+S_{max}$, where $S_{max}=\max_{s\in S} s$ is the largest step size.
By transitivity and since no intersection can occur until the first time $Y_j\geq X_0$ then it actually suffices to verify for the case $X_0=0 \leq Y_0 < S_{max}$.

A natural approach to studying collisions is to consider an appropriate random variable counting the number of intersections of the two walks.
Towards this, let $S_N$ denote the number of times the $X_i$ walk intersects the $Y_j$ walk in the first $N$ steps, i.e.
$$
S_N = \sum_{i=0}^N \bone_{\{\exists j : \,X_i=Y_j\}}\,.
$$

If one intersection is unlikely to be followed soon by others then $\Pr{S_N>0}\approx \EE(S_N)$.
To measure the gap between the two quantities, let $B_{\epsilon}$ be the worst-case expected number of collisions between two independent walks before the nearly uniform intersection time $T(\epsilon)$. To be precise:
$$
B_{\epsilon} = \max_{Y_0 < X_0+S_{max}} E \sum_{i=1}^{T(\epsilon)} \bone_{\{\exists j : \,X_i=Y_j\}}
$$

The main result of this section bounds the expected number of steps until a collision.

\begin{theorem}\label{thm:birthday}
Given an increasing transitive Markov chain on $\ZZ$, if two independent walks have starting states with $Y_0<X_0+S_{max}$ then
\begin{eqnarray*}
\EE \min\{i>0:\,\exists j,\,X_i=Y_j\}
      &\leq& 1+\left(\frac{\sqrt{\bar{S}(1+B_{\epsilon})}+\sqrt{\overline{T(\epsilon)}}}{1-\epsilon}\right)^2 \\
\EE \min\{i>0:\,\exists j,\,X_i=Y_j\}
      &\geq& 1+\bar{S}\,\frac{\left(\max\{0,\,1-\sqrt{B_{\epsilon}}\}\right)^2}{1+\epsilon}
\end{eqnarray*}
\end{theorem}

In particular, when $\epsilon$ and $B_\epsilon$ are close to zero and $\bar{S}\gg \overline{T(\epsilon)}$ then
$$
\EE \min\{i>0:\,\exists j,\,X_i=Y_j\} \sim \bar{S}\,,
$$
which makes rigorous the heuristic that the expected number of steps needed until a collision is the average step size.

The steps before a nearly uniform intersection time act as a sort of burn-in period, so it will be easier if we discard them in the analysis. As such, let
$$
\S_\Delta =\sum_{i=T(\epsilon)+1}^{T(\epsilon)+\Delta} \bone_{\{\exists j:\,X_i=Y_j\}}\,.
$$

The first step in the proof is to examine the number of collisions after the burn-in:

\begin{lemma} \label{lem:expectations}
Under the conditions of Theorem \ref{thm:birthday}, if $\Delta\geq 0$ then
\begin{eqnarray*}
(1-\epsilon)\,\frac{\Delta}{\bar{S}} \leq& \displaystyle E[\S_\Delta] &\leq (1+\epsilon)\,\frac{\Delta}{\bar{S}} \\
& \displaystyle E[\S_\Delta \mid \S_\Delta>0] &\leq 1 + B_{\epsilon} + E[\S_\Delta\mid X_0=Y_0=0]
\end{eqnarray*}
\end{lemma}

\begin{proof}
The expectation $E[\S_\Delta]$ satisfies
\begin{eqnarray*}
E[\S_\Delta]
 &=&    E \sum_{i=1}^{\Delta} \bone_{\{\exists j:\,X_{T(\epsilon)+i}=Y_j\}} \\
 &=&    \sum_{i=1}^{\Delta} \Pr{\exists j:\,X_{T(\epsilon)+i}=Y_j}  \\
 &\geq& \Delta\,\frac{1-\epsilon}{\bar{S}}
\end{eqnarray*}
The upper bound on $E[\S_\Delta]$ follows by taking $(1+\epsilon)$ in place of $(1-\epsilon)$.

Now for $E[\S_\Delta \mid \S_\Delta>0]$.
Observe that if $X_i=Y_j$ and $k>i$ then $X_k=Y_\ell$ can occur only for $\ell>j$, because the $X$ and $Y$ walks are increasing.
Hence, if $\tau=\min\{i>0\,:\,\exists j,\,X_{T(\epsilon)+i}=Y_j\}$ is the time of the first intersection, the number of intersections after time $\tau$ can be found by considering the case $X_0=Y_0$ and then computing the expected number of intersections until $X_{T(\epsilon)+\Delta-i}$.
The total number of intersections is then
\begin{eqnarray*}
\lefteqn{E[\S_\Delta \mid \S_\Delta>0]} \\
 &=& \sum_{i=1}^{\Delta} \Pr{\tau=T(\epsilon)+i}\left(1+E_{X_0=Y_0} \sum_{k=1}^{\Delta-i} \bone_{\{\exists \ell:\,X_k=Y_\ell\}}\right) \\
 &\leq& 1 + E_{X_0=Y_0} \sum_{k=1}^{T(\epsilon)} \bone_{\{\exists \ell:\,X_k=Y_\ell\}} + E_{X_0=Y_0} \sum_{k=T(\epsilon)+1}^{T(\epsilon)+\Delta} \bone_{\{\exists \ell:\,X_k=Y_\ell\}} \\
 &\leq& 1 + B_{\epsilon} + E[\S_\Delta\mid X_0=Y_0=0]
\end{eqnarray*}
\end{proof}

This shows that if $B_{\epsilon}$ is small then one intersection is rarely followed by others, or more rigorously:
\begin{lemma} \label{lem:prob}
Under the conditions of Theorem \ref{thm:birthday}, if $\Delta\geq 0$ then
\begin{eqnarray*}
\Pr{S_{T(\epsilon)+\Delta}>0} &\leq& \frac{\Delta}{\bar{S}}\,(1+\epsilon) + B_{\epsilon} \\
\Pr{S_{T(\epsilon)+\Delta}>0} &\geq& \frac{\Delta}{\bar{S}}\,\frac{(1-\epsilon)^2}{1+B_{\epsilon}+\frac{\Delta}{\bar{S}}}\,.
\end{eqnarray*}
\end{lemma}

\begin{proof}
Observe that a random variable $Z\geq 0$ satisfies
\begin{equation}\label{eqn:positive}
\Pr{Z>0} = \frac{E[Z]}{E[Z\mid Z>0]}
\end{equation}
because $E[Z] = \Pr{Z=0}\,E[Z \mid Z=0] + \Pr{Z>0}\,E[Z \mid Z>0 ]$.

For the lower bound let $Z=\S_\Delta$ in \eqref{eqn:positive}, so that
\begin{eqnarray*}
\Pr{S_{T(\epsilon)+\Delta}>0}&\geq&\Pr{\S_\Delta>0}
  \geq \frac{E[\S_\Delta]}{1 + B_{\epsilon} + \max E[\S_\Delta]} \\
  &\geq& \frac{(1-\epsilon)\Delta/\bar{S}}{1 + B_{\epsilon} + (1+\epsilon)\Delta/\bar{S}}
   \geq \frac{1-\epsilon}{1+\epsilon}\,\frac{\Delta/\bar{S}}{\frac{1+B_{\epsilon}}{1+\epsilon}+\Delta/\bar{S}}\,.
\end{eqnarray*}

For the upper bound take $Z=S_{T(\epsilon)+\Delta}$ in \eqref{eqn:positive}, so that
$$
\Pr{S_{T(\epsilon)+\Delta}>0} = \frac{E[S_{T(\epsilon)+\Delta}]}{E[S_{T(\epsilon)+\Delta}\mid S_{T(\epsilon)+\Delta}>0]}\leq E [S_{T(\epsilon)+\Delta}] \,.
$$
Since $Y_0>X_0$ then the expectation $E[S_{T(\epsilon)+\Delta}]$ satisfies
\begin{eqnarray*}
E[S_{T(\epsilon)+\Delta}]
  &=&    E \sum_{i=0}^{T(\epsilon)+\Delta} \bone_{\{\exists j:\,X_i=Y_j\}} \\
  &=& E \sum_{i=1}^{T(\epsilon)} \bone_{\{\exists j:\,X_i=Y_j\}} + \sum_{i=T(\epsilon)+1}^{T(\epsilon)+\Delta} \bone_{\{\exists j:\,X_i=Y_j\}} \\
 &\leq& B_{\epsilon} + \Delta\,\frac{1+\epsilon}{\bar{S}} \,.
\end{eqnarray*}
\end{proof}

\begin{proof}[Proof of Theorem~\ref{thm:birthday}]
The walk will be broken into blocks of length $T(\epsilon)+\Delta$ for some $\Delta$ to be optimized later, overlapping only at the endpoints, and each block analyzed separately.

More formally, inductively define $N_0=0$, let $T_k(\epsilon)$ be the nearly uniform intersection time started at state $X_{N_{k-1}}$, and set $N_k=N_{k-1}+T_k(\epsilon)+\Delta$.
The number of intersections from time $N_k$ to $N_{k+1}$ is
$$
S_{N_k}^{N_{k+1}} = \sum_{i=N_k}^{N_{k+1}} \bone_{\{\exists j : \,X_i=Y_j\}}\,.
$$
By taking $X_0 \leftarrow X_{N_k}$ and $Y_0\leftarrow \min\{Y_j:\, Y_j\geq X_{N_k}$\} then Lemma \ref{lem:prob} implies
$$
\frac{\Delta}{\bar{S}}(1+\epsilon) + B_{\epsilon}
 \geq \Pr{S_{N_k}^{N_{k+1}}>0\,\mid\, S_{N_k}=0}
 \geq \frac{\Delta}{\bar{S}} \frac{(1-\epsilon)^2}{1+B_{\epsilon}+\frac{\Delta}{\bar{S}}}\,.
$$
Since
\begin{eqnarray*}
\Pr{S_{N_\ell}=0}
  &=& \prod_{k=0}^{\ell-1} \Pr{S_{N_k}^{N_{k+1}}=0\mid S_{N_k=0}}
\end{eqnarray*}
then
$$
\left(1-\frac{\Delta}{\bar{S}}\,\frac{(1-\epsilon)^2}{1+B_{\epsilon}+\frac{\Delta}{\bar{S}}}\right)^\ell
  \geq \Pr{S_{N_\ell}=0}
  \geq \left(1-B_{\epsilon}-\frac{\Delta}{\bar{S}}\,(1+\epsilon) \right)^\ell\,.
$$

The blocks will now be combined to prove the theorem.

First, the upper bound.
\begin{eqnarray*}
E \min\{i:\,S_i>0\}-1
    &=& E \sum_{i=0}^\infty \bone_{\{S_i=0\}}-1
     =  \sum_{k=0}^\infty E\sum_{i=N_k+1}^{N_{k+1}} \bone_{\{S_i=0\}} \\
    &=& \sum_{k=0}^\infty \Pr{S_{N_k}=0}\,E\left[\sum_{i=N_k+1}^{N_{k+1}} \bone_{\{S_i=0\}}\,{\Bigr |}\,S_{N_k}=0\right] \\
    &\leq&  \sum_{k=0}^\infty \left(1-\frac{\Delta}{\bar{S}}\,\frac{(1-\epsilon)^2}{1+B_{\epsilon}+\frac{\Delta}{\bar{S}}}\right)^k\,\left(\overline{T(\epsilon)}+\Delta\right) \\
     &=&    \frac{\bar{S}}{\Delta}\,\frac{1+B_{\epsilon}+\frac{\Delta}{\bar{S}}}{(1-\epsilon)^2}\,\left(\overline{T(\epsilon)}+\Delta\right)
\end{eqnarray*}
This is minimized when $\Delta = \sqrt{\bar{S}(1+B_{\epsilon})\overline{T(\epsilon)}}$.

The lower bound is similar.
\begin{eqnarray*}
\lefteqn{ E \min\{i:\,S_i>0\}-1
    =  \sum_{k=0}^\infty E\sum_{i=N_k+1}^{N_{k+1}} \bone_{\{S_i=0\}} } \\
    &\geq& \sum_{k=0}^\infty \Pr{S_{N_{k+1}}=0}\,E\left[\sum_{i=N_k+1}^{N_{k+1}} \bone_{\{S_i=0\}}\,{\Bigr |}\,S_{N_{k+1}}=0\right] \\
    &\geq& \sum_{k=0}^\infty \left(1-B_{\epsilon}-\frac{\Delta}{\bar{S}}\,(1+\epsilon)\right)^{k+1}\,\Delta \\
    &=&    \left(\frac{1}{B_{\epsilon}+\frac{\Delta}{\bar{S}}\,(1+\epsilon)}-1\right)\,\Delta
\end{eqnarray*}
This is maximized when $\Delta=\max\left\{0,\frac{\sqrt{B}(1-\sqrt{B})\bar{S}}{1+\epsilon}\right\}$.
\end{proof}

The following lemma makes it possible to bound $B_{\epsilon}$ given bounds on multi-step transition probabilities.

\begin{lemma} \label{lem:B_epsilon}
If $T(\epsilon)$ is a nearly uniform intersection time then
$$
B_{\epsilon} \leq \sum_{i=1}^M (1+2i)\max_{u,v}\P^i(u,v) + M\,\left(\frac{2(S_{max}/\bar{S})^2}{\bar{S}}(1+\epsilon) + e^{-M}\right)\,.
$$
\end{lemma}

\begin{remark} \label{rmk:bound}
To apply the lemma in the unbounded case observe that if $M$ is a constant then $T'(\epsilon')=\min\{T(\epsilon),M\}$ is a bounded nearly uniform intersection time with $\epsilon'=\epsilon+\frac{\Pr{T(\epsilon)>M}}{1/\bar{S}}$.
\end{remark}

\begin{proof}
If $Y_0<X_0$ then no intersections can occur until the first time $Y_j\geq X_0$ so the maximum in the definition of $B_\epsilon$ is achieved by some $Y_0\geq X_0$, i.e.
$$
B_{\epsilon} = \max_{X_0\leq Y_0<X_0+S_{max}} E\sum_{i=1}^{T(\epsilon)} \bone_{\{\exists j:\,X_i=Y_j\}}\,.
$$

The $\{Y_j\}$ walk will be examined in three pieces: a burn-in, a mid-range, and an asymptotic portion.
In particular, since $T(\epsilon)\leq M$ then for any constant $N\geq M$
$$
B_\epsilon
 \leq  \max_{X_0\leq Y_0 < X_0+S_{max}}
			E\sum_{i=1}^M \left(\sum_{j=0}^M \bone_{\{X_i=Y_j\}} + \sum_{j=M+1}^N \bone_{\{X_i=Y_j\}} + \bone_{\{\exists j>N:\,X_i=Y_j\}}\right)
$$

Consider the first summation.
\begin{eqnarray*}
E \sum_{i=1}^M \sum_{j=0}^M \bone_{\{X_i=Y_j\}}
  &=& E \sum_{i=1}^M \sum_{j=0}^M \sum_w \bone_{\{X_i=Y_j=w\}} \\
  &=& \sum_{i=1}^M \sum_{j=0}^M \sum_w \P^i(X_0,w)\P^j(Y_0,w) \\
  &\leq& \sum_{i=1}^M \max_{u,v} \P^i(u,v)\,\sum_{j=0}^i (1+\bone_{\{j<i\}}) \max_z \sum_w \P^j(z,w) \\
  &=&    \sum_{i=1}^M (1+2i)\max_{u,v} \P^i(u,v)
\end{eqnarray*}
The inequality follows by letting $i$ denote the larger of the two indices and $j$ the smaller,
while the final equality is because $\sum_w\P^j(z,w)=1$.

Next, the case when $M<j\leq N$.
\begin{eqnarray*}
E \sum_{i=1}^M \sum_{j=M+1}^N \bone_{\{X_i=Y_j\}}
  &=& E \sum_{j=M+1}^N \sum_{i=1}^M \bone_{\{X_i=Y_j\}} \\
  &=& \sum_{j=M+1}^N \Pr{\exists i\in[1\ldots M]:\,X_i=Y_j} \\
  &\leq& E \sum_{j=T_Y(\epsilon)+1}^{T_Y(\epsilon)+N} \Pr{\exists i:\,X_i=Y_j} \\
  &\leq& N\frac{1+\epsilon}{\bar{S}}
\end{eqnarray*}
The second equality is because the $X_i$ walk is increasing, so for fixed $j$ there can be at most one $i$ with $X_i=Y_j$. The first inequality is because $X_i$ and $Y_j$ are instances of the same Markov Chain, and so the stopping time $T(\epsilon)$ induces a nearly uniform intersection time $T_Y(\epsilon)\leq M$ for the $Y_j$ walk as well.
This applies as long as $X_0<Y_0+S_{max}$ as is the case here.

Finally the case that $j>N$. By Hoeffding's Inequality
$$
\Pr{Y_N-Y_0\leq \frac 12\,N\bar{S}}
  \leq \exp\left(\frac{-N^2\bar{S}^2}{2N\,S_{max}^2}\right)
    =  \exp\left(-\frac 12\,N\left(\bar{S}/S_{max}\right)^2\right)\,.
$$
Set $N=2M(S_{max}/\bar{S})^2$. Then with probability $1-e^{-M}$
$$
Y_N > Y_0 + \frac 12\,N\bar{S}
   \geq Y_0+M\,S_{max}
   \geq X_0+M\,S_{max}
   \geq X_M\,.
$$
In particular, $\Pr{Y_N\leq X_M} \leq e^{-M}$ and so
$$
E \sum_{i=1}^M \bone_{\{\exists j>N:\,X_i=Y_j\}}
  \leq M\,\Pr{Y_N\leq X_M} \leq M\,e^{-M}\,.
$$
\end{proof}


\section{Catching Kangaroos}

The tools developed in the previous section will now be applied to a concrete problem, Pollard's Kangaroo Method for discrete logarithm.


\subsection{Pollard's Kangaroo Method} \label{sec:prelim}

We describe here the Kangaroo method, originally known as the Lambda method for catching Kangaroos.
The Distinguished Points implementation of \cite{VW99.1} is given because it is more efficient than the original implementation of \cite{Pol78.1}.

\vspace{2ex}\noindent{\bf Problem:} {\em Given $g,h\in G$, solve for $x\in[a,b]$ with $h=g^x$.}

\vspace{1ex}\noindent{\bf Method:} Pollard's Kangaroo method (distinguished points version).

\vspace{1ex}\noindent{\bf Preliminary Steps:}
\begin{itemize}
\item Define a set $D\subset G$ of ``distinguished points'', with $\frac{|D|}{|G|}=\frac{c}{\sqrt{b-a}}$ for some constant $c$.
\item Define a set of jump sizes $S=\{s_0,s_1,\ldots,s_d\}$.
We consider powers of two, $S=\{2^k\}_{k=0}^d$, with
$d\approx\log_2 \sqrt{b-a} + \log_2\log_2 \sqrt{b-a}-2$, chosen so that elements of $S$ average to a jump size of $\bar{S}\approx\frac{\sqrt{b-a}}{2}$.
This can be made an equality by taking $p:\,S\to[0,1]$ to be a probability distribution such that $\bar{S}=\sum_{s\in S} s\,p(s)=\frac{\sqrt{b-a}}{2}$.
\item Finally, a hash function $F:G\to S$
which ``randomly'' assigns jump sizes such that $\Pr{F(\g)=s}\approx p(s)$ for every $\g\in G$.
\end{itemize}

\vspace{1ex}\noindent{\bf The Algorithm:}
\begin{itemize}
\item Let $Y_0=\frac{a+b}{2}$, $X_0=x$, and $d_0=0$. Observe that $g^{X_0}=hg^{d_0}$.
\item Recursively define $Y_{j+1}=Y_j+F(g^{Y_j})$ and likewise $d_{i+1}=d_i+F(hg^{d_i})$. This implicitly defines $X_{i+1}=X_i+F(g^{X_i})=x+d_{i+1}$.
\item If $g^{Y_j}\in D$ then store the pair $(g^{Y_j},Y_j-Y_0)$ with an identifier $T$ (for tame). Likewise if $g^{X_i}=hg^{d_i}\in D$ then store $(g^{X_i},d_i)$ with an identifier $W$ (for wild).
\item Once some distinguished point has been stored with both identifiers $T$ and $W$, say $g^{X_i}=g^{Y_j}$ where $(g^{X_i},d_j)$ and $(g^{Y_j},Y_j-Y_0)$ were stored, then
\begin{eqnarray*}
Y_j\equiv X_i\equiv x+d_i \mod |G| && \\
\Longrightarrow\ x\equiv Y_j-d_i \mod|G| &&
\end{eqnarray*}
\end{itemize}

The $Y_j$ walk is called the ``tame kangaroo'' because its position is known, whereas the position $X_i$ of the ``wild kangaroo'' is to be determined by the algorithm.
This was originally known as the Lambda method because the two walks are initially different, but once $g^{Y_j}=g^{X_i}$ then they proceed along the same route, forming a $\lambda$ shape.

Theorem~\ref{thm:main} makes rigorous the following commonly used heuristic:
Suppose $X_0\in[a,b]$ is a uniform random value and $Y_0\geq X_0$.
Run the tame kangaroo infinitely far.
The wild kangaroo requires $\EE(Y_0-X_0)/\bar{S}=(b-a)/(4\bar{S})$ steps to reach $Y_0$.
Subsequently, at each step the probability that the wild kangaroo lands on a spot visited by the tame kangaroo is roughly $\wp=\bar{S}^{-1}$,
so the expected number of additional steps by the wild kangaroo until a collision is then around $\wp^{-1}=\bar{S}$.
By symmetry the tame kangaroo also averaged $\wp^{-1}$ steps until a collision.
About $\frac{\sqrt{b-a}}{c}$ additional steps are required until a distinguished point is reached.
Since $X_i$ and $Y_j$ are incremented simultaneously the total number of steps taken is then
$$
2\left(\frac{b-a}{4\bar{S}} + \bar{S} + \frac{\sqrt{b-a}}{c}\right)\,.
$$
This is minimized when $\bar{S}=\frac{\sqrt{b-a}}{2}$, with $(2+2c^{-1})\sqrt{b-a}$ steps sufficing.

If, instead, the distribution of $X_0$ is unknown then in the worst case $|Y_0-X_0|/\bar{S} = (b-a)/(2\bar{S})$ and the bound is $(3+2c^{-1})\sqrt{b-a}$ when $\bar{S}=\frac{\sqrt{b-a}}{2}$.

Our analysis assumes that the Kangaroo method involves a truly random hash function:
if $\g\in G$ then $F(\g)$ is equally likely to be any of the jump sizes, independent of all other $F(\g')$.
In practice different hash functions will be used on different groups -- whether over a
subgroup of integers mod p, elliptic curve groups, etc -- but in general the hash is chosen to ``look random.''
Since the Kangaroo method applies on all cyclic groups then a constructive proof would involve the impossible task of explicitly constructing a hash
on every cyclic group, and so the assumption of a truly random hash is made in all attempts at analyzing it of which we are aware \cite{Tes01.1,VW99.1,Pol00.1}.
A second assumption is that the distinguished points are well distributed with $c\xrightarrow{(b-a)\to\infty}\infty$; either they are chosen uniformly at random, or if $c=\Omega(d^2\log d)$ then roughly constant spacing between points will suffice.
The assumption on distinguished points can be dropped if one instead analyzes Pollard's (slower) original algorithm, to which our methods also apply.


\subsection{Analysis of the Kangaroo Method}  \label{sec:kangaroo}

In order to understand our approach to bounding time until the kangaroos have visited a common location, which we call a {\em collision}, it will be helpful to consider a simplified version of the Kangaroo method.
First, observe that because hash values $F(\g)$ are independent then $X_i$ and $Y_j$ are independent random walks at least until they intersect, and so to bound time until this occurs it suffices to assume they are independent random walks even after they have collided.
Second, these are random walks on $\ZZ/|G|\ZZ$, so if we drop the modular arithmetic and work on $\ZZ$ then the time until a collision can only be made worse.
Third, since the walks proceed strictly in the positive direction on $\ZZ$ then in order to determine the number of hops the ``wild kangaroo'' (described by $X_i$) takes until it is caught by the ``tame kangaroo'' (i.e. $X_i=Y_j$ on $\ZZ$),
it suffices to run the tame kangaroo infinitely long and only after this have the wild kangaroo start hopping.

The intersection results of the previous section will now be applied to the Kangaroo method.
Recall that $d$ is chosen so that the average step size in $S=\{2^k\}_{k=0}^d$ is roughly $\bar{S}\approx\frac{\sqrt{b-a}}{2}$,
and that this can be made an equality by choosing step sizes from a probability distribution $p$ on $S$.
In this section we analyze the natural setting where $\frac{\gamma}{d+1}\geq p(s)\geq\frac{\gamma^{-1}}{d+1}$ for some constant $\gamma\geq 1$; indeed $\gamma=2$ is sufficient for some $p,\,d$ to exist with $\bar{S}=\frac{\sqrt{b-a}}{2}$ exactly.

The first step in bounding collision time will be to construct a nearly uniform intersection time.
Our approach involves constructing a tentative stopping time $\T_{tent}$ where $Y_{\T_{tent}}$ is uniformly distributed over some interval of length $L$,
and then accepting or rejecting this in such a way that $Y_j$ will be equally likely to visit any state beyond the left endpoint of the interval in which it is first accepted.
It follows that once $X_i\geq Y_\T$ then the probability that $X_i=Y_j$ for some $j$ will be a constant.

\begin{lemma} \label{lem:intersection_time}
Consider a Kangaroo walk with step sizes $S=\{2^k\}_{k=0}^d$ and transition probabilities $\frac{\gamma}{d+1}\geq p(s)\geq\frac{\gamma^{-1}}{d+1}$ for some constant $\gamma\geq 1$.
Then there is a bounded nearly uniform intersection time with $\Omega=\{(X_0,Y_0):\,|X_0-Y_0|<S_{max}=2^d\}$ and
$$
T\left(\frac{2}{d+1}\right)\leq 64\gamma^5(d+1)^5\,.
$$
\end{lemma}

\begin{proof}
Consider a lazy walk $\tilde{Y}_t$ with $\tilde{Y}_0=Y_0$ in which a step consists of choosing an item $s\in S$ according to $p$, and then half the time make the transition $u\to u+s$ and half the time do nothing.
The probability that this walk eventually visits a given state $y\in\ZZ$ is exactly the same as for the $Y_j$ walk, so it suffices to replace $Y_j$ by $\tilde{Y}_t$ when showing a nearly uniform intersection time.

For each $s\in S$ let $\delta_s$ denote the step size taken the first time $s$ is chosen, so that $\Pr{\delta_s=0}=\Pr{\delta_s=s}=1/2$.
Define a tentative stopping time $\T_{tent}$ by stopping the first time every $s\in S-\{2^d\}=\{2^k\}_{k=0}^{d-1}$ has been chosen at least once.
Observe that $\delta:=\sum_{s\in S-\{2^d\}} \delta_s\in\{0,1,\ldots, 2^d-1\}$ uniformly at random.
Accept the stopping time with probability $\sum_{s\in S:\,s>\delta} p(s)$ and set $\T=\T_{tent}$.
If it is rejected then re-initialize all $\delta_s$ values (and $\delta$) and continue the $\tilde{Y}_t$ walk until a new stopping time is determined, which can again be either accepted or rejected.

Observe that $\delta\in\{0,1,\ldots,2^d-1\}$ has distribution $\Pr{\delta=\ell}\propto \sum_{s>\ell} p(s)$.
The normalization factor is $\sum_{\ell=0}^{2^d-1}\sum_{s>\ell}p(s) = \sum_{s\in S} p(s)s = \bar{S}$
and so the distribution is
$$
\Pr{\delta=\ell} = \frac{\sum_{s>\ell}p(s)}{\bar{S}}
$$

This stopping rule was constructed so that if $y\geq\tilde{Y}_\T-\delta$ then, as will now be shown,
$\Pr{\exists t:\,\tilde{Y}_t=y}=\bar{S}^{-1}$,
making $T=\min\{i:\,X_i\geq \tilde{Y}_\T-\delta\}$ a uniform intersection time for $X_i$.

Suppose $y=\I_\T$.
The quantity $\I_\T:=\tilde{Y}_\T-\delta$ is independent of $\delta$ because it depends only on those steps not included in a $\delta_s$.
It follows that $\Pr{\exists t:\,\tilde{Y}_t=y\,\mid\,\I_\T}=\Pr{\delta=0}=\bar{S}^{-1}$.

If $y>\I_\T$ then inductively assume that $\Pr{\exists t:\,\tilde{Y}_t=v}=\bar{S}^{-1}$ for all $v\in[\I_\T,y)$.
Then
\begin{eqnarray*}
\lefteqn{\Pr{\exists t:\,\tilde{Y}_t=y\,\mid\,\I_\T} } \\
  &=& \Pr{\delta=y-\I_\T\,\mid\,\I_\T} + \sum_{\I_\T\leq v<y} \Pr{\exists t:\,\tilde{Y}_t=v\,\mid\,\I_\T}\,p(y-v) \\
  &=& \frac{\sum_{v<\I_\T} p(y-v)}{\bar{S}} + \frac{\sum_{\I_\T\leq v<y} p(y-v)}{\bar{S}} = \frac{1}{\bar{S}}
\end{eqnarray*}

It remains only to compute $T=\min\{i:\,X_i\geq \tilde{Y}_\T-\delta\}$, which in turn requires a value for $\T$.

To determine time until a tentative stopping time $\T_{tent}$ it suffices to find the probability that in $j$ steps some generator in $\{2^k\}_{k=0}^d$ has not been chosen.
The probability a specified generator is chosen in step $j$ is at least $\frac{\gamma^{-1}}{d+1}$, and so
\begin{eqnarray*}
\Pr{\T_{tent}>j}
  &\leq& \sum_{k=0}^d \Pr{s=2^k \textrm{ has not been chosen in $j$ steps}} \\
  &\leq& (d+1)\left(1-\frac{\gamma^{-1}}{d+1}\right)^j \\
  &\leq& (d+1)e^{-j/\gamma(d+1)}
\end{eqnarray*}
As a result
$$
\Pr{\T_{tent} \geq \gamma(d+1)\ln(2\gamma(d+1)^2)} \leq \frac{1}{2\gamma(d+1)}
$$
Each tentative stopping time is accepted with probability
$$
\sum_{\delta=0}^{2^d-1} \frac{1}{2^d}\times\sum_{s>\delta} p(s)
   = \frac{\sum_s p(s)s}{2^d}
   = \frac{\bar{S}}{2^d}
   \geq \frac{1}{2^d}\,\frac{\gamma^{-1}S_{max}}{d+1} = \frac{1}{\gamma(d+1)}
$$
It follows that in $2\gamma(d+1)\ln(1/\epsilon)$ rounds of $\gamma(d+1)\ln(2\gamma(d+1)^2)$ steps each the probability that a stopping time has not yet been determined is then at most
$$
\left(\frac{1}{2\gamma(d+1)} + 1- \frac{1}{\gamma(d+1)}\right)^{2\gamma(d+1)\ln(1/\epsilon)}
 \leq \epsilon
$$
and so if $\T(\epsilon)=2\gamma^2(d+1)^2\ln(2\gamma(d+1)^2)\ln(1/\epsilon)$ then
$\Pr{\T > \T(\epsilon)} \leq \epsilon$.

Finally, it remains to determine $T(\epsilon)$.
If $X_i\geq \tilde{Y}_{\T(1/(d+1)\bar{S})}$ then
$$
\big| \Pr{\exists t:\,X_i=\tilde{Y}_t}-\bar{S}^{-1} \big| \leq \frac{1}{(d+1)\bar{S}}
$$
By Lemma \ref{lem:Hoeffding} below, if
$M = 4\gamma^3(d+1)^3\ln(2\gamma(d+1)^2)\ln((d+1)\bar{S})$
then
$$
\Pr{X_M < \tilde{Y}_{\T(1/(d+1)\bar{S})}} \leq 1/[(d+1)\bar{S}]
$$
and so overall $\big| \Pr{\exists t:\,X_i=\tilde{Y}_\T}-\bar{S}^{-1} \big| \leq 2/[(d+1)\bar{S}]$.
It follows that
$$
T\left(\frac{2}{d+1}\right) \leq M = 4\gamma^3(d+1)^3\ln(2\gamma(d+1)^2)\ln((d+1)\bar{S})
$$
This simplifies via the relations $\ln(x)\leq x$ and
$$
\bar{S} = \sum_{k=0}^d 2^k\,p(2^k) \leq \frac{\gamma}{d+1}\,\sum_{k=0}^d 2^k<\frac{2\gamma}{d+1}\,2^d
$$
\end{proof}

The following simple application of Hoeffding's Inequality was used above.

\begin{lemma} \label{lem:Hoeffding}
Suppose a non-negative random variable has average $\bar{S}$ and maximum $S_{max}$.
If $N$ is a constant and $\delta_1,\,\delta_2,\,\ldots,\,\delta_M$ are some $M$ independent samples then the sum $X=\sum_{i=1}^M \delta_i$ satisfies
$$
\Pr{X < (1+N)S_{max}} \leq \epsilon
$$
when
$$
M = 2\frac{S_{max}}{\bar{S}}\,
\max\left\{ \frac{S_{max}}{\bar{S}}\ln(1/\epsilon),\,1+N\right\}
$$
\end{lemma}

\begin{proof}
Recall Hoeffding's Inequality, that if $Y$ is the sum of $n$ independent random variables with values in $[a,b]$ then for any $t\geq 0$
$$
\Pr{Y -\EE Y \geq t} \leq \exp\left(\frac{-2t^2}{n(b-a)^2}\right)\,.
$$
Taking $Y=-X$ as the sum of $-\delta_i\in[-S_{max},0]$ it follows that
$$
\Pr{X-\EE X \leq -\frac{M}{2}\,\bar{S}} \leq \exp\left(\frac{-M^2\bar{S}^2}{2M\,S_{max}^2}\right)
$$
Plugging in $\EE X = M\bar{S}$ with $M$ from the Lemma finishes the proof.
\end{proof}

It remains only to upper bound $B_{\epsilon}$.

\begin{lemma} \label{lem:B_epsilon-kangaroo}
The nearly uniform intersection time of Lemma \ref{lem:intersection_time}
has
$$
B_{\epsilon}=\Theta\left(\frac{1}{d+1}\right)=o_d(1)
$$
\end{lemma}

\begin{proof}
This will be shown by applying Lemmas \ref{lem:B_epsilon} and \ref{lem:intersection_time}.

First consider the walk $\hat{\P}$ where $\gamma=1$, i.e. step sizes are chosen uniformly at random.
Observe that $\hat{\P}^i(u,v)=\frac{c_i(u,v)}{(d+1)^i}$ where $c_i(u,v)$ is the number of ways to write $v-u$ as the sum of $i$ (non-distinct, ordered) elements of $\{2^k\}_{k=0}^d$.
In the binary expansion of $v-u$ a non-zero bit $2^\ell$ can only arise as the sum of at most $i$ steps chosen from $\{2^k\}_{k=\ell-i+1}^\ell$, and so any string of more than $i-1$ consecutive zeros can be contracted to $i-1$ zeros without effecting the number of ways to write $v-u$.
This shows that $c_i=\max_{u,v} c_i(u,v)$ can be determined by considering only the bit strings $v-u$ of length $i^2$,
and in particular it is upper bounded by a constant independent of $d$, i.e. $\hat{\P}^i(u,v)=O((d+1)^{-i})$.

In the non-uniform case $\P^i(u,v)\leq \gamma^i\,\hat{\P}^i(u,v) \leq \frac{c_i\gamma^i}{(d+1)^i}$.

If $i\geq 12$ then
\begin{eqnarray*}
\max_{u,v}\P^i(u,v)
  &=& \max_{u,v} \sum_w \P^{i-12}(u,w)\P^{12}(w,v) \\
  &\leq& \max_{u,v} \sum_w \P^{i-12}(u,w)\max_{w}\P^{12}(w,v) \\
  &=& \max_{w,v} \P^{12}(w,v) \leq \frac{c_{12}\gamma^{12}}{(d+1)^{12}}\,.
\end{eqnarray*}

Hence, with $M=64\gamma^5(d+1)^5$ then
\begin{eqnarray*}
\lefteqn{ \sum_{i=1}^M (1+2i)\max_{u,v}\P^i(u,v) } \\
 &\leq& \frac{3*\gamma^1}{d+1}
      +\sum_{i=2}^{11} \frac{(1+2i)*c_i\gamma^i}{(d+1)^i}
      +\frac{(1+2M)(M-11)c_{12}\gamma^{12}}{(d+1)^{12}} \\
 &=& \frac{3\gamma+O(1/(d+1))}{d+1}\,.
\end{eqnarray*}

A bound of $B_\epsilon = \frac{3\gamma+o_d(1)}{d+1}$ follows by applying Lemma \ref{lem:B_epsilon} with $S_{max}=2^d$ and
$$
\bar{S}=\sum_{k=0}^d 2^k\,p(2^k) \geq \frac{\gamma^{-1}}{d+1}\,\sum_{k=0}^d 2^k>\frac{\gamma^{-1}}{d+1}\,2^d\,.
$$

For a corresponding lower bound let $X_0=Y_0$ so that $B_\epsilon \geq \Pr{X_1=Y_1}=\sum_{s\in S}p(s)^2$.
By Cauchy-Schwarz
$$
1=\sum_{s\in S} p(s)\times 1 \leq \sqrt{\sum_{s\in S} p(s)^2}\sqrt{\sum_{s\in S} 1^2}
$$
and so $\sum_{s\in S} p(s)^2\geq \frac{1}{|S|}=\frac{1}{d+1}$ and $B_\epsilon \geq \frac{1}{d+1}$.
\end{proof}

All the tools are now in place to prove the main result of the paper.

\begin{proof}[Proof of Theorem \ref{thm:main}]
Note that the group elements $g^{(2^k)}$ can be pre-computed, so that each step of a kangaroo requires only a single group multiplication.

As discussed in the heuristic argument of Section~\ref{sec:prelim}, an average of $\frac{|Y_0-X_0|}{\bar{S}}$ steps are needed to put the smaller of the starting states (e.g. $Y_0<X_0$) within $S_{max}=2^d$ of the one that started ahead.
If the Distinguished Points are uniformly randomly distributed then the heuristic for these points is again correct.
If instead they are roughly constantly spaced and $c=\Omega(d^2\log d)$ then observe that, in the proof of Lemma~\ref{lem:intersection_time} it was established that after some $\T_{tent}$
steps the kangaroos will be uniformly random over some interval of length $2^d\sim \frac 14\sqrt{b-a}\log_2\sqrt{b-a}$.
It is easily seen that $\EE\T_{tent}\leq\gamma(d+1)^2$,
so if the Distinguished Points cover a $\frac{c}{\sqrt{b-a}}$ fraction of vertices then an average of $\frac{\sqrt{b-a}}{c}$ such samples are needed, independent of $\T_{tent}$.
It follows that an average of $\EE(\T_{tent})\frac{\sqrt{b-a}}{c}=o_d(1)*\sqrt{b-a}$ extra steps suffice.

It remains to make rigorous the claim regarding $\wp^{-1}$.
In the remainder we may thus assume that $|Y_0-X_0|< 2^d=S_{max}$.
By Lemma~\ref{lem:intersection_time} a bounded nearly uniform intersection time has
$T\left(\frac{2}{d+1}\right)\leq 64\gamma^5(d+1)^5$,
while Lemma \ref{lem:B_epsilon-kangaroo} shows that $B_{\epsilon}=o_d(1)$.
The upper bound of Theorem \ref{thm:birthday} is then $\left(\frac 12+o_d(1)\right)\sqrt{b-a}$
while the lower bound is $\left(\frac 12-o_d(1)\right)\sqrt{b-a}$.
\end{proof}


\section{Resolution of a Conjecture of Pollard} \label{sec:generalize}

In the previous section the Kangaroo method was analyzed for the most common situation, when the generating set is given by powers of $2$.
Pollard conjectured in \cite{Pol00.1} that the same result holds for powers of any integer $n\geq 2$, again under the assumption that $\bar{S}\approx \frac{\sqrt{b-a}}{2}$.
In this section we show his conjecture to be correct.

\begin{theorem}
When step sizes are chosen from $S=\{n^k\}_{k=0}^d$ with transition probabilities $\frac{\gamma}{d+1}\geq p(s)\geq\frac{\gamma^{-1}}{d+1}$ such that $\bar{S}=\frac{\sqrt{b-a}}{2}$,
then Theorem \ref{thm:main} still holds.
\end{theorem}

\begin{proof}
We detail only the differences from the case when $n=2$.

To construct a nearly uniform intersection time, once again consider the walk $\tilde{Y}_t$ which half the time does nothing.
Partition the steps into blocks of $(n-1)$ consecutive steps each.
If the same generator $s\in\{n^k\}_{k=0}^{d-1}$ is chosen at every step in a block then let $m$ be the number of times a step of size $s$ was taken (recall it's lazy), so $\Pr{m=\ell}=\frac{\binom{n-1}{\ell}}{2^{n-1}}$, and with probability $\binom{n-1}{m}^{-1}$ set $\delta_s=m\,s$ if $\delta_s$ is undefined.
In all other cases do not change any $\delta_s$ after the $(n-1)$ steps have been made.
Stop when every $\delta_s$ has been defined.

Observe that $\delta_s$ is uniformly chosen from the possible values $\{m\,s\}_{m=0}^{n-1}$, so the sum $\delta=\sum_s \delta_s$ is a uniformly random $d$ digit number in base $n$.
Once again, accept this candidate stopping time with probability $\sum_{s\in S:\,s>\delta}p(s)$,
and otherwise reset the $\delta_s$ values and find another candidate stopping time.
The same proof as before verifies that if $\I_t:=\tilde{Y}_\T-\delta$ then
$\Pr{\exists t:\,\tilde{Y}_t=y\,\mid\,y\geq\I_\T}=1/\bar{S}$.

Next, determine the number of steps required for the $\tilde{Y}_t$ walk to reach this stopping time.
First consider the time required until a tentative stopping time $\T_{tent}$.
For a specified block and generator $s=n^k$, the probability $s$ was chosen at every step in the block is at least $\left(\frac{\gamma^{-1}}{d+1}\right)^{n-1}$, and when this happens the probability the resulting value is accepted is
$$
\sum_{m=0}^{n-1} \frac{\binom{n-1}{m}}{2^{n-1}}\times \frac{1}{\binom{n-1}{m}} = \frac{n}{2^{n-1}}\,.
$$
Combining these quantities, if $\delta_s$ was previously undefined then the probability it is assigned a value in this block is
$$
\wp\geq\left(\frac{\gamma^{-1}}{d+1}\right)^{n-1}\frac{n}{2^{n-1}}\,.
$$

The probability of not stopping within $N(n-1)$ steps is then
\begin{eqnarray*}
\Pr{\T_{tent}>N(n-1)}
  &\leq& (d+1)(1-\wp)^N \\
  &\leq& (d+1)\exp\left(-\frac{Nn}{(2\gamma(d+1))^{n-1}}\right)
\end{eqnarray*}
and so when $N=(2\gamma(d+1))^{n-1}n^{-1}\ln(2\gamma(d+1)^2)$ then this shows
$$
\Pr{\T_{tent}\geq (2\gamma(d+1))^{n-1}\ln(2\gamma(d+1)^2)}\leq \frac{1}{2\gamma(d+1)}
$$

The remaining calculations are not specific to the base $2$ case and so they carry through smoothly,
leading to a nearly uniform intersection time of
$$
T\left(\frac{2}{d+1}\right) = 2(2\gamma(d+1))^{n+3}\ln n
$$

To extend Lemma \ref{lem:B_epsilon-kangaroo} replace $2^k$ by $n^k$ throughout,
use $M=O\left((d+1)^{n+3}\right)$, and bound large powers of $\P$ in terms of $\P^{2n+8}$ instead of $\P^{12}$.
This results in $B_\epsilon=\Theta(1/(d+1))$ again.

The proof of Theorem \ref{thm:main} carries through with only obvious adjustments.
\end{proof}


\section*{Acknowledgements}

The authors thank Dan Boneh for encouraging them to study the Kangaroo method,
and John Pollard for several helpful comments.


A preliminary version appeared in the Proceedings of $41^{\textrm{st}}$ ACM Symposium on Theory of Computing (STOC 2009).
This journal version extends the notion of nearly uniform intersection time to the setting of stopping times,
contains simplifications which led to sharper results,
and on the application side the solution to Pollard's conjecture given in Section \ref{sec:generalize} is entirely new.


\end{document}